\documentclass[12pt]{amsart}
\usepackage{graphicx, amsfonts, amsmath, amssymb, cleveref, esvect, xcolor, nicematrix} 

\newtheorem{theorem}{Theorem}[section]

\newtheorem*{maintheorem}{Main Theorem}

\newtheorem{lemma}[theorem]{Lemma}

\newtheorem{claim}[theorem]{Claim}

\theoremstyle{definition}

\numberwithin{equation}{section}

\title[A lower bound on end-periodic stretch factors]{A lower bound on end-periodic stretch factors}

\author[M. Loving and C. Wu]{Marissa Loving and Chenxi Wu}

\date{\today}

\begin{document}

\begin{abstract}
    Given an end-periodic homeomorphism $f: S \to S$ we give a lower bound on the Handel--Miller stretch factor of $f$ in terms of the \emph{core characteristic of $f$}, which is a measure of topological complexity for an end-periodic homeomorphism. We also show that the growth rate of this bound is sharp.
\end{abstract}

\maketitle

\section{Introduction}

In this note we consider the Handel--Miller stretch factor $\lambda_{HM}$ associated to an end-periodic homeomorphism $f$. Our main result is a lower bound on $\lambda_{HM}$ analogous to Penner's lower bound on the least dilatation of pseudo-Anosov homeomorphisms \cite{penner1991bounds}. Here $\chi(f)$ is the \emph{core characteristic of $f$} which we define in \Cref{sub:endperiodic}.

\begin{maintheorem} \label{thm:main}
    Let $f: S \to S$ be an end-periodic homeomorphism. If $\lambda_{HM}\not=1,$ then $\log\lambda_{HM}\geq {\frac{\log(2)}{3\chi(f)}}$.
\end{maintheorem}

Penner's proof utilizes train track theory, to bound the complexity of the Markov decomposition associated to a pseudo-Anosov homeomorphism of a closed surface, together with a spectral theory trick \cite[Lemma]{penner1991bounds}, to bound the leading eigenvalue, which we will also employ here. However, we are able to obtain bounds on the complexity of the Markov decomposition for an end-periodic homeomorphism much more directly by utilizing the fact that $\Lambda^+ \cap \Lambda^-$ is contained in some compact subsurface $C \subset S$ such that $\Lambda^+ \cap C$ (resp. $\Lambda^- \cap C$) is a collection of arcs. 

It was pointed out to us by Sam Taylor that the Handel--Miller stretch factor $\lambda_{HM}$ that we consider here coincides with the stretch factor of spun pseudo-Anosov maps described in \cite[Theorem~C]{LandryMinskyTaylor2023}. We include a brief discussion of this after the proof of the Main Theorem in \Cref{section:proof}. 

\section{Preliminaries}

\subsection{End-periodic homeomorphisms} \label{sub:endperiodic} Throughout we will let $S$ denote a connected, orientable surface of infinite type with finitely many ends all of which are non-planar. An \emph{end-periodic homeomorphism} $f: S \to S$ is a homeomorphism satisfying the following. There exists $m>0$ such that for each end $E$ of $S$, there is a neighborhood $U_E$ of $E$ so that either 
    \begin{enumerate}
        \item[(i)] $f^m(U_E) \subsetneq U_E$ and the sets ${f^{nm}(U_E)}_{n > 0}$ form a neighborhood basis of $E$; or
        \item[(ii)] $f^{-m}(U_E) \subsetneq U_E$ and the sets ${f^{-nm}(U_E)}_{n > 0}$ form a neighborhood basis of $E$.
    \end{enumerate}
Following \cite{EndPeriodic1}, we call such a neighborhood $U_E$ a \emph{nesting neighborhood}. 

End-periodic homeomorphisms were introduced and studied in unpublished work of Handel--Miller. In particular, they developed a robust lamination theory for (irreducible) end-periodic homeomorphisms analogous to the lamination theory for pseudo-Anosov homeomorphisms in the finite-type setting. For a detailed treatment of this lamination theory see \cite{CC-book}. For our purposes the following discussion suffices. 

Fix an end-periodic homeomorphism $f: S \to S$. Let $U_+$ be a union of nesting neighborhoods of the attracting ends and let $U_-$ be a collection of nesting neighborhoods of the repelling ends. Following the language used in \cite{CC-book} and \cite{Whitfield} we call the set $U_+$ (resp. $U_-$) a \emph{positive} (resp. \emph{negative}) \emph{ladder} of $f$. In addition, we say the ladder $U_{\pm}$ is \emph{tight} if $f^{\pm 1}(U_{\pm}) \subset U_{\pm}$ is a proper inclusion. A disjoint pair of positive and negative tight ladders $U_+, U_-$ define a compact subsurface $Y = S - (U_+ \cup U_-)$, which is called a \emph{core} for $f$.

Following \cite{EndPeriodic2}, we define the \emph{core characteristic} of $f$ to be \[\chi(f) = \max_{Y \subset S} \chi(Y),\] where the maximum is taken over all cores $Y \subset S$ for $f$. Any core $Y$ with $\chi(Y) = \chi(f)$ is a \emph{minimal core}.

Define \[\mathcal U_+ = \bigcup_{n \geq 0} f^{-n}(U_+)\text{ and } \mathcal U_- = \bigcup_{n \geq 0} f^n(U_-).\] We call $\mathcal U_+$ the \emph{positive escaping set} and $\mathcal U_-$ the \emph{negative escaping set}. 

An essential multiloop in $\mathcal U_+$ (resp. $\mathcal U_-$) is a \emph{positive} (resp. \emph{negative}) \emph{juncture} of $f$ if it is the boundary of a tight positive (resp. negative) ladder. Given a positive and negative juncture $j_+$ and $j_-$ of $f$ we define \[J^+ = \bigcup_{k \in \mathbb Z} f^k(j_+) \text{ and } J^- = \bigcup_{k \in \mathbb Z} f^k(j_-).\]

Fix a hyperbolic metric $X$ on $S$ and let $\mathcal J^{\pm}$ be the union of tightened geodesics for each curve in $J^{\pm}$. Note that $\overline{\mathcal J^{\pm}}$ is a geodesic lamination on $S$ since it is the union of disjoint simple geodesics. The \emph{Handel--Miller laminations} associated to $f$ are the geodesic laminations defined by \[\Lambda^+ = \overline{\mathcal J^-} - \mathcal J^- \text{ and } \Lambda^- = \overline{\mathcal J^+} - \mathcal J^+.\] Note that, by definition, $\mathcal J^-$ is disjoint from $\Lambda^+$ and $\mathcal J^+$ is disjoint from $\Lambda^-$.

The intersection of the Handel--Miller laminations $\Lambda^{+} \cap \Lambda^-$ gives a Markov decomposition of the complement of the escaping points, and we denote by $\lambda_{HM}$ the corresponding spectral radius of the incidence matrix, which is the exponential of the topological entropy on the action of $f$ on $\Lambda^+\cap\Lambda^-$. We call $\lambda_{HM}$ the \emph{Handel--Miller stretch factor}. Note that the number of rectangles contained in this Markov decomposition depends on the structure of $\Lambda^{\pm}$. 

\subsection{A key lemma} 

An important tool in our proof of the Main Theorem is the Lemma from \cite{penner1991bounds}. Recall that a matrix $M$ is Perron--Frobenius (abbreviated PF) if $M$ has nonnegative entries and some power of $M$ is strictly positive. Here we weaken the hypothesis that the integral matrix $A$ is PF to only require that $A$ is nonnegative with leading eigenvalue $> 1$. We do this since the incidence matrix of the Markov decomposition induced by $\Lambda^+ \cap \Lambda^-$ is not necessarily PF, i.e. if the invariant measures on $\Lambda^{\pm}$ do not have full support. We state and sketch the proof of this generalized Penner lemma below.  

\begin{lemma}[Generalized Penner Lemma] \label{lemma:penner}
    Suppose that $A$ is an $n \times n$ nonnegative integral matrix, where $n> 1$, with leading eigenvalue $> 1$. If $\lambda$ denotes the spectral radius of $A$, then \[\log \lambda \geq \frac{\log 2}{n}.\]
\end{lemma}

\begin{proof} 
    Given an $n \times n$ matrix $A$ with nonnegative entries, we define the \emph{oriented graph} $G = G(A)$ to be the graph on $n$ vertices with an oriented edge running from the $i^{th}$ vertex to the $j^{th}$ vertex if $A_{ij} \neq 0$. We call $A$ the \emph{transition matrix} of $G$. It is easy to see that the $(i, j)$-th entry of $A^m$, denoted as $A_{ij}^m$, is exactly the number of oriented edge-paths in $G$ of length $m$ from the $i^{th}$ to the $j^{th}$ vertex.

    Note that there exists some strongly connected subgraph $G'$ of $G$ whose transition matrix has the same leading eigenvalue as $A$. This is because the set of eigenvalues of a block triangular matrix is equal to the union of the eigenvalues of its diagonal blocks. So, in particular, the leading eigenvalue must correspond to an eigenvalue of some diagonal block. Because the leading eigenvalue is $>1$, this subgraph $G'$ is not itself a circle. Let $n' = |V(G')|$ and denote by $B$ the transition matrix associated to $G'$. 

    Suppose that the vector $\vv{x} \in \mathbb R^{n'}$ has coordinates $(x_i)_1^{n'},$ and consider the function \[|\vv{x}| = \sum_{i=1}^{n'} x_i\] defined on the ``cone" \[C = \{\vv{x} \in \mathbb R^{n'} :  x_i \geq 0 \text{ for each $i$ and } \vv{x} \neq \vv{0}\}.\]

    As in Penner's proof, in order to show that $\log \lambda(B) \geq \frac{\log 2}{n},$ it suffices to show that for any unit coordinate vector $\vv{e} \in C$, we have \[|B^{n'}\vv{e}| \geq 2.\]

    Since we know that $G'$ is not a circle, for each vertex $v \in V(G')$ there are at lesst two paths of length $n'$ starting from $v$. Thus, $|B^{n'} \vv{e}| \geq 2$ and because $n \geq n'$, we are done.
\end{proof}

\section{The proof} \label{section:proof}

\begin{proof}[Proof of Main Theorem]

Fix a minimal core $Y$ of $f$. The intersection of the leaves of $\Lambda^+$ with $Y$ consists only of arcs (see \cite[Lemma~4.52]{CC-book}). We can modify $Y$ so that all these intersections are essential arcs on $Y$, by which we mean that they are not isotopic to an arc on the boundary of $Y$. Up to isotopy, there are only finitely many possible non-intersecting essential arcs, and the maximal number of such arcs, $N_Y$, is bounded by $-3\chi(Y)$, which can be seen by doubling $Y$ across its boundary and counting non-isotopic loops. For each of these arcs $\alpha_i$, $f^{-1}(\alpha_i) \cap Y$ is a union of arcs in $\Lambda^+$. Let $m_{ij}$ be the number of such arcs isotopic to $\alpha_j$. Then the isotopy types of arcs in $\Lambda^+$ gives a Markov decomposition of $\Lambda^+\cap \Lambda^-$, and the corresponding incidence matrix is $M=[m_{ij}]$ which is a nonnegative integer matrix. Thus, $\log\lambda_{HM}$ is the $\log$ of its maximal eigenvalue, which, when non-zero, is bounded from below by $\log 2/N_Y$ by \Cref{lemma:penner}. \qedhere

\end{proof}

Note that the Handel--Miller stretch factor $\lambda_{HM}$ we consider here coincides with the stretch factor $\lambda_{spA}$ of a spun pseudo-Anosov map introduced by Landry--Minsky--Taylor in \cite{LandryMinskyTaylor2023}. This was brought to our attention by Sam Taylor along with the following discussion. Consider the intersection of the Handel--Miller laminations $\Lambda^+ \cap \Lambda^-$. In \cite{LandryMinskyTaylor2023}, they define the core dynamical system $C_f$ associated to a spun pseudo-Anosov map $f$, whose topological entropy is $\log(\lambda_{spA})$. In addition, they show that there is a semiconjugacy (with finite fibers) between $\Lambda_{spA}^+ \cap \Lambda_{spA}^-$ and $C_f$ \cite[Section~5.2]{LandryMinskyTaylor2023}, where $\Lambda_{spA}^{\pm}$ are invariant sublaminations of the invariant singular foliations associated to a spun pseudo-Anosov map. By \cite[Theorem~8.4]{LandryMinskyTaylor2023}, $\Lambda_{spA}^{\pm}$ have the same ideal endpoints as the Handel--Miller laminations $\Lambda^{\pm}$. Since the semiconjugacy mentioned above is proved using the ideal endpoints of $\Lambda_{spA}^{\pm}$, there is also a semiconjugacy (with finite fibers) between $\Lambda^+ \cap \Lambda^-$ and $C_f$. Thus, $\lambda_{HM} = \lambda_{spA}$.

\section{A sequence of examples}

Finally, we give a sequence of examples of end-periodic homeomorphisms which extend Penner's construction \cite{penner-construction} to the infinite-type setting. We construct these examples using a quotient of the Loch Ness monster surface (that is, the surface with exactly one end accumulated by genus). Recall that Penner's construction produces pseudo-Anosov homeomorphisms that can be obtained as the composition of a rotation of the surface with the product of (positive and negative) Dehn twists around simple closed curves filling a proper subsurface. This construction is often described as a phone dial, see \Cref{fig:PhoneDial}. 

\begin{figure}[htb!]
    \centering
    \includegraphics[width=0.3\linewidth]{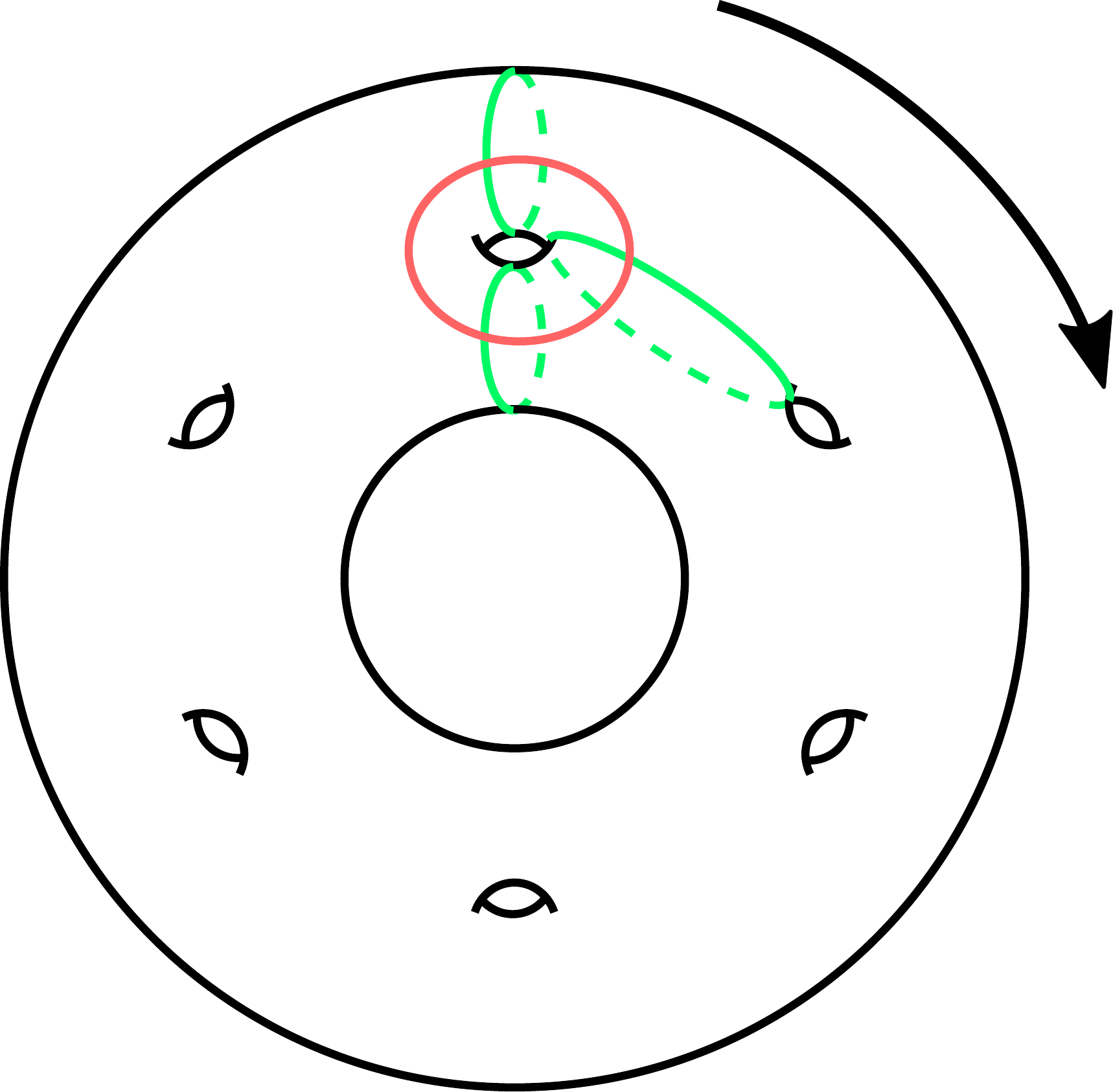}
    \caption{Note that after finitely many rotations the green and pink curves define a filling pair of multicurves.}
    \label{fig:PhoneDial}
\end{figure}

Let $S$ be the Loch Ness monster surface and consider the representation of $S$ shown in \Cref{fig:WaffleSurface}. We replace the rigid rotation shown in \Cref{fig:PhoneDial} with a composition of shifts in the ``horizontal" and ``vertical" directions. Let $p_d$ be the deck transformation given by composing a vertical shift by $d$ steps with a horizontal shift by $1$ step and let $S_d$ be the quotient of $S$ by the subgroup generated by $p_d$. Define $f_d: S_d \to S_d$ to be the map given by the Dehn twist about the pink curve on $S_d$ shown in \Cref{fig:EndPeriodicExample} composed with a shift by $1$ step in the vertical direction. Its $d$-th iterate, $f_d^d$, is now the map described in \Cref{fig:EndPeriodicExample}. Note that $f_d^d: S_d \to S_d$ coincides with the lift of the map $f_1: S_1 \to S_1$ to its $d$-fold cover. Thus, $f_d^d$ and $f_1$ have the same stretch factor and so, 
\[\log\lambda_{HM}(f_d)=\frac{\log\lambda_{HM}(f_1)}{d}.\]
We chose $f_d$ so that $f_1:S_1 \to S_1$ nearly coincides with the map described in \cite[Example~4.13]{CC-book} (simply double along all boundary components while twisting in one component and shifting in both). In order to conclude that the growth rate of the estimate in the Main Theorem is sharp, it only remains to show that $\lambda_{HM}(f_1) > 1$. We will do this in the following claim.

\begin{claim}
    Let $f_1:S_1 \to S_1$ be the map and $\tau$ the invariant train track illustrated in \Cref{fig:Iteration}. Then $\lambda_{HM}(f_1) = 2$.
\end{claim}

\begin{proof}[Proof of Claim]
    The action of $f_1$ on $\tau$ can be described by the following linear operator using the choice of basis indicated by the labeling on branches shown in \Cref{fig:Iteration}. 

\[
    \begin{pNiceArray}{cc | cccc}
      0 & 1 & 0 & 0 & 0 & \cdots \\
      0 & 2 & 0 & 0 & 0 & \cdots \\
      \hline
      1 & 0 & 0 & 0 & 0 & \cdots \\
      0 & 1 & 0 & 0 & 0 & \cdots \\
      0 & 0 & 1 & 0 & 0 & \cdots \\
      \vdots & \vdots & \vdots & \vdots & \vdots & \ddots \\
    \end{pNiceArray}
\] 

\medskip

Note that the top two rows are entirely $0$ after the first two entries while the remaining rows are entirely $0$ except for a $1$ in the $(i-2)$-th column of the $i$-th row. This operator has leading eigenvalue $2$. Hence, $\lambda_{HM}(f_1) = 2$, as desired.\end{proof}

\begin{figure}
    \def\svgwidth{6in}
\begingroup%
  \makeatletter%
  \providecommand\color[2][]{%
    \errmessage{(Inkscape) Color is used for the text in Inkscape, but the package 'color.sty' is not loaded}%
    \renewcommand\color[2][]{}%
  }%
  \providecommand\transparent[1]{%
    \errmessage{(Inkscape) Transparency is used (non-zero) for the text in Inkscape, but the package 'transparent.sty' is not loaded}%
    \renewcommand\transparent[1]{}%
  }%
  \providecommand\rotatebox[2]{#2}%
  \newcommand*\fsize{\dimexpr\f@size pt\relax}%
  \newcommand*\lineheight[1]{\fontsize{\fsize}{#1\fsize}\selectfont}%
  \ifx\svgwidth\undefined%
    \setlength{\unitlength}{1590.71874584bp}%
    \ifx\svgscale\undefined%
      \relax%
    \else%
      \setlength{\unitlength}{\unitlength * \real{\svgscale}}%
    \fi%
  \else%
    \setlength{\unitlength}{\svgwidth}%
  \fi%
  \global\let\svgwidth\undefined%
  \global\let\svgscale\undefined%
  \makeatother%
  \begin{picture}(1,0.60016319)%
    \lineheight{1}%
    \setlength\tabcolsep{0pt}%
    \put(0,0){\includegraphics[width=\unitlength,page=1]{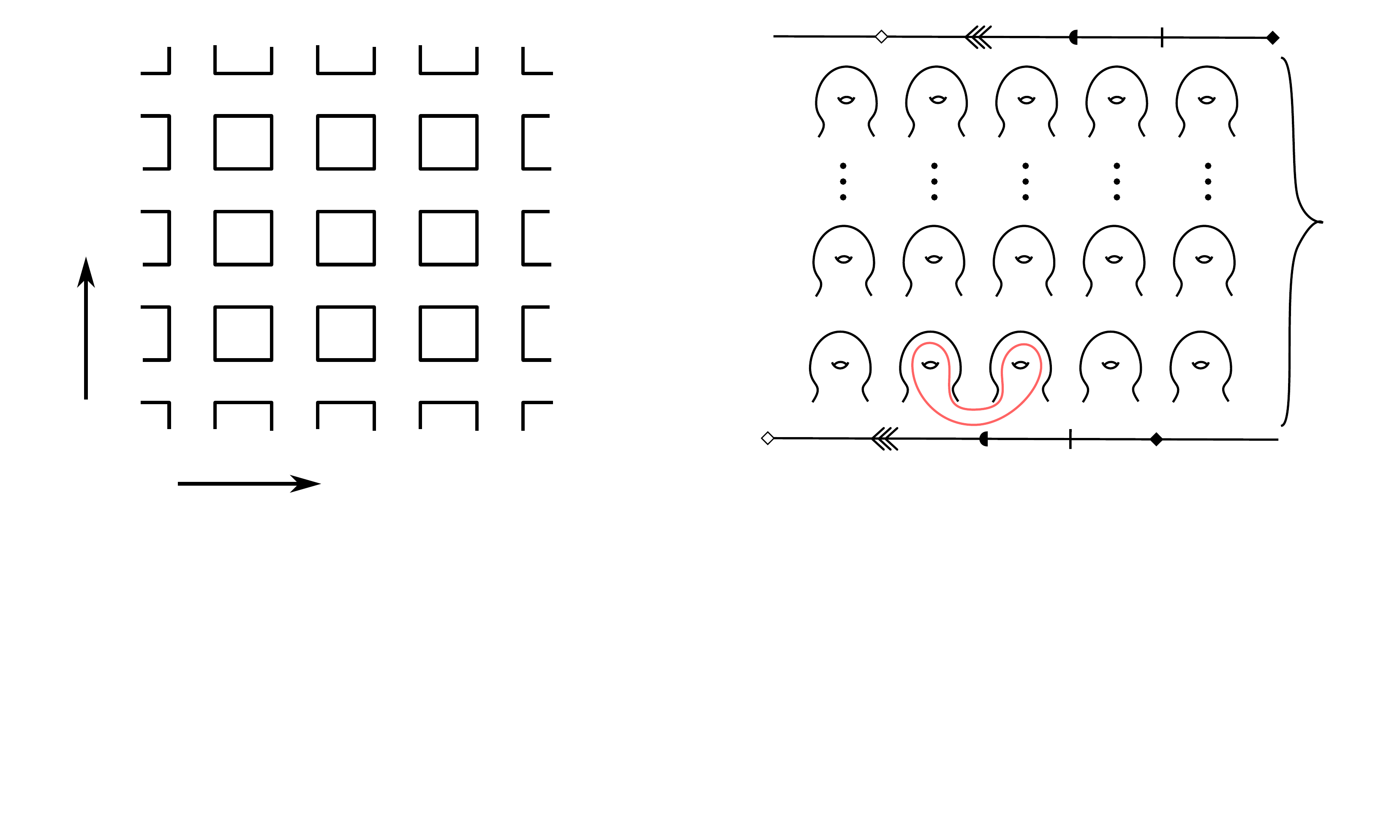}}%
    \put(0.04120705,0.35178411){\makebox(0,0)[rt]{\lineheight{1.25}\smash{\begin{tabular}[t]{r}$d$\end{tabular}}}}%
    \put(0.17407004,0.20381665){\makebox(0,0)[t]{\lineheight{1.25}\smash{\begin{tabular}[t]{c}$1$\end{tabular}}}}%
    \put(0,0){\includegraphics[width=\unitlength,page=2]{Covers.pdf}}%
    \put(0.95867139,0.43862532){\makebox(0,0)[lt]{\lineheight{1.25}\smash{\begin{tabular}[t]{l}$d$\end{tabular}}}}%
    \put(0.09518164,0.57762805){\makebox(0,0)[rt]{\lineheight{1.25}\smash{\begin{tabular}[t]{r}$S$\end{tabular}}}}%
    \put(0.54943096,0.58274772){\makebox(0,0)[rt]{\lineheight{1.25}\smash{\begin{tabular}[t]{r}$S_d$\end{tabular}}}}%
    \put(0.26844881,0.10745359){\makebox(0,0)[rt]{\lineheight{1.25}\smash{\begin{tabular}[t]{r}$S_1$\end{tabular}}}}%
    \put(0,0){\includegraphics[width=\unitlength,page=3]{Covers.pdf}}%
  \end{picture}%
\endgroup%

    \caption{The Loch Ness monster surface, $S$, is an infinite cyclic cover of $S_d$ which is in turn a $d$-fold cover of $S_1$. Note $S_d$ and $S_1$ are, in fact, homeomorphic. The horizontal and vertical shifts shown on $S$ are used to define the deck transformation $p_d$.}
    \label{fig:WaffleSurface}
\end{figure}

\begin{figure}
    \def\svgwidth{3.5in}
\begingroup%
  \makeatletter%
  \providecommand\color[2][]{%
    \errmessage{(Inkscape) Color is used for the text in Inkscape, but the package 'color.sty' is not loaded}%
    \renewcommand\color[2][]{}%
  }%
  \providecommand\transparent[1]{%
    \errmessage{(Inkscape) Transparency is used (non-zero) for the text in Inkscape, but the package 'transparent.sty' is not loaded}%
    \renewcommand\transparent[1]{}%
  }%
  \providecommand\rotatebox[2]{#2}%
  \newcommand*\fsize{\dimexpr\f@size pt\relax}%
  \newcommand*\lineheight[1]{\fontsize{\fsize}{#1\fsize}\selectfont}%
  \ifx\svgwidth\undefined%
    \setlength{\unitlength}{580.26496396bp}%
    \ifx\svgscale\undefined%
      \relax%
    \else%
      \setlength{\unitlength}{\unitlength * \real{\svgscale}}%
    \fi%
  \else%
    \setlength{\unitlength}{\svgwidth}%
  \fi%
  \global\let\svgwidth\undefined%
  \global\let\svgscale\undefined%
  \makeatother%
  \begin{picture}(1,0.71877684)%
    \lineheight{1}%
    \setlength\tabcolsep{0pt}%
    \put(0,0){\includegraphics[width=\unitlength,page=1]{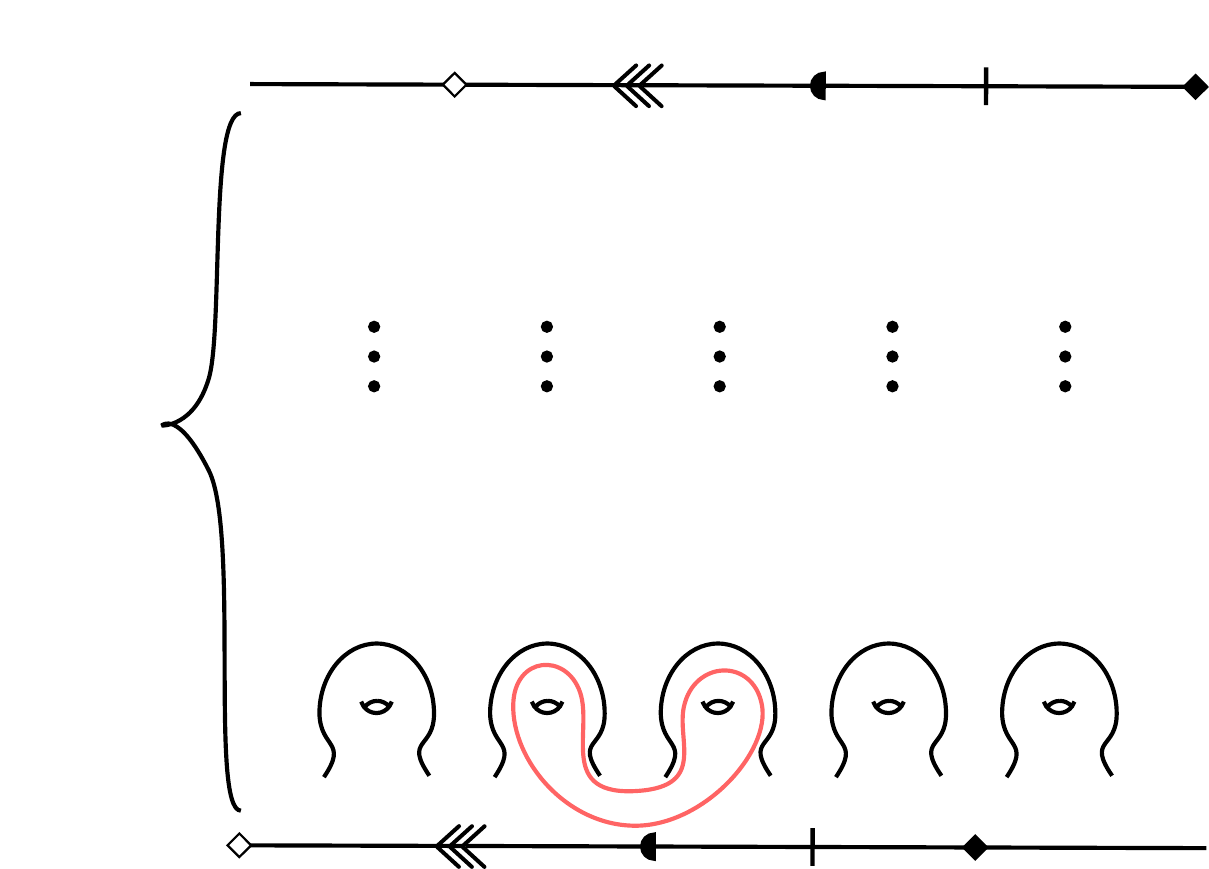}}%
    \put(0.11729286,0.36756605){\makebox(0,0)[rt]{\lineheight{1.25}\smash{\begin{tabular}[t]{r}$d$\end{tabular}}}}%
    \put(0.19487663,0.67103464){\makebox(0,0)[rt]{\lineheight{1.25}\smash{\begin{tabular}[t]{r}$S_d$\end{tabular}}}}%
    \put(0,0){\includegraphics[width=\unitlength,page=2]{EndperiodicExample.pdf}}%
  \end{picture}%
\endgroup%

     \caption{The map $f_d^d: S_d \to S_d$, is the composition of the twists about the pink curves composed with a shift by step $1$ along each row in the horizontal direction.}
    \label{fig:EndPeriodicExample}
\end{figure}

\begin{figure}
    \def\svgwidth{4.5in}
    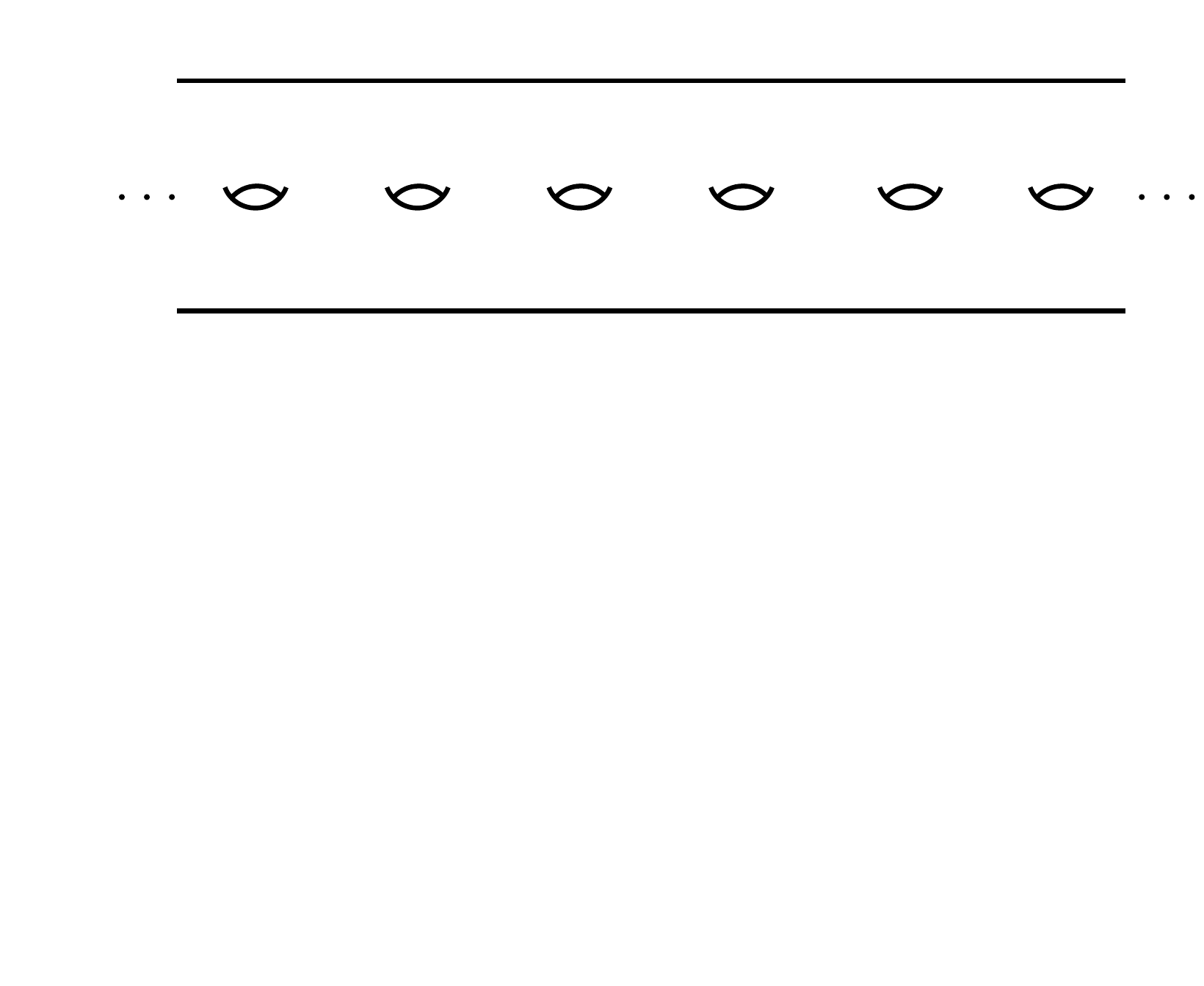
     \caption{Consider the map $f_1 = D \circ h$, where $D$ is the Dehn twist about the curve labeled $D$ and $h$ is the shift map which translates by $1$ to the right. The train track $\tau$ is invariant under $f_1$.} 
    \label{fig:Iteration}
\end{figure}

\section*{Acknowledgments} We are grateful to Sam Taylor and the anonymous referee for their careful reading and comments on a draft of this paper. Loving was supported in part by NSF Grant DMS-2231286.

\newpage

\bibliographystyle{alpha}
  \bibliography{main}

\end{document}